\documentclass[11pt,a4paper]{amsart}
\usepackage{amsmath}
\usepackage{amsthm}
\usepackage{amssymb}
\usepackage{url}
\usepackage{amsfonts}
\usepackage{anysize}
\marginsize{2.5cm}{2.5cm}{2.5cm}{3.4cm}

\usepackage{amscd}        
\usepackage{graphics}
\usepackage{color}
\usepackage{amsthm}   
\usepackage{mathrsfs}
\usepackage[all]{xy}
\usepackage{textcomp}
\usepackage{mathrsfs}

\numberwithin{equation}{section} \theoremstyle{plain}
\newtheorem{theorem}{Theorem}[section]
\newtheorem{lemma}[theorem]{Lemma}
\newtheorem{corollary}[theorem]{Corollary}
\newtheorem{proposition}[theorem]{Proposition}
\theoremstyle{definition}
\newtheorem{definition}[theorem]{Definition}

\theoremstyle{remark}
\newtheorem{remark}[theorem]{Remark}

\def\Ho{\mbox{Ho}}
\def\O{{\cal O}}
\def\Z{\mathbb Z}
\def\mcF{\mathcal F}
\def\mcB{\mathcal B}

\def\Qco{\mathfrak{Qcoh}}

\def\Mod{\mbox{-Mod}}

\def\bfC+{\mathbf{C}_+}

\def\F{\tilde{\mathcal{F}}}
\def\C{{\mathcal{C}}}

\def\bfC{\mathbf{C}}

\def\O{\mathcal{O}}
\def\scrM{\mathscr{M}}

\def\scrA{\mathscr{A}}
\def\scrF{\mathscr{F}}

\def\scrG{\mathscr{G}}

\newcommand{\Ch}{\mathbf{C}}

\newcommand{\im}{\operatorname{Im}}

\newcommand{\Hom}{\operatorname{Hom}}
\newcommand{\Filt}{\mbox{Filt}}

\newcommand{\id}{\operatorname{id}}

\newcommand{\Proj}{\lambda\mbox{-Pproj}}
\newcommand{\Pproj}{\lambda\operatorname{-Pproj}}

\newcommand{\coker}{\mbox{Coker}}

\newcommand{\Inj}{\operatorname{Inj}}
\newcommand{\Pinj}{\operatorname{\otimes-Pinj}}

\newcommand{\dg}{{\mathit dg\,}}

\begin{document}

\title{Pure exact structures and the pure derived category of a scheme }

\author{Sergio Estrada}
\address{Departamento de Matem\'aticas \\
         Universidad de Murcia \\
         Campus del Espinardo\\
         30100 Espinardo, Murcia (Spain)}
\email[Sergio Estrada]{sestrada@um.es}
\urladdr{http://webs.um.es/~sestrada/}

\author{James Gillespie}
\address{Ramapo College of New Jersey \\
         School of Theoretical and Applied Science \\
         505 Ramapo Valley Road \\
         Mahwah, NJ 07430}
\email[Jim Gillespie]{jgillesp@ramapo.edu}
\urladdr{http://pages.ramapo.edu/~jgillesp/}
\author{S\.{I}nem Odaba\c{s}i}
\address{Departamento de Matem\'aticas \\
         Universidad de Murcia \\
         Campus del Espinardo\\
         30100 Espinardo, Murcia (Spain)}
\email[Sinem Odaba\c{s}{\i}]{sinem.odabasi1@um.es}

\date{\today}

\begin{abstract}
Let $\mathcal C$ be closed symmetric monoidal Grothendieck category. We define the pure derived category with respect to the monoidal structure via a relative injective model category structure on the category $\Ch(\C)$ of unbounded chain complexes in $\C$. We use $\lambda$-Purity techniques to get this. As application we define the stalkwise pure derived category of the category of quasi--coherent sheaves on a quasi-separated scheme. We also give a different approach by using the category of flat quasi--coherent sheaves. 

\end{abstract}

\maketitle
\section{Introduction}
In \cite{CB} Crawley-Boevey showed that locally finitely presented additive categories are the natural framework to define a good notion of Purity Theory. We recall that a locally finitely presented additive category $\mathcal A$ is an additive category with direct limits such that every object is a direct limit of finitely presented objects, and the class of finitely presented objects is skeletally small. Then a sequence $0\to M\to N\to T\to 0$ in $\mathcal A$ is pure if $0\to \Hom(G,M)\to \Hom(G,N)\to \Hom(G,T)\to 0$ is exact for each finitely presented object $G$ in $\mathcal A$. This defines a pure exact structure in $\mathcal A$ and yields the pure derived category $\mathcal{D}_{\rm pur}(\mathcal A)$ studied for example by Christensen and Hovey \cite{CH} and Krause \cite{Krause}. Recently in \cite{G} it has been shown that this pure derived category can be obtained as the homotopy category of two model category structures by using the pure projectives and the pure injectives. Locally finitely presented categories are quite abundant in Algebra as they include module categories, but also in Algebraic Geometry as most of the schemes that occur in practice (e.g. quasi-compact and quasi-separated schemes) are such that the category of quasi-coherent sheaves is always a locally finitely presented Grothendieck category. However, unless the scheme $X$ is affine, the {\it categorical} purity defined above for locally finitely presented categories does not coincide with the local purity on $\Qco(X)$. Let us see this in more detail. It makes sense to define purity of a short exact sequence in $\Qco(X)$ in terms of the purity of the corresponding short exact sequences on its stalks. And this seems to be a more reasonable way of defining purity in $\Qco(X)$ as it reflects the local nature of the definition. Let us call this definition {\it geometric} purity. Needless to say, geometric purity and categorical purity agree when $X$ is affine. But in case $X$ is not affine the two notions are different, and in general categorical purity implies geometric purity but the converse need not be true. Namely, assume that $X$ is quasi-compact and semi-separated and let us call an $F$ in $\Qco(X)$ a categorical flat sheaf provided that every short exact sequence $0\to N\to M\to F\to 0$ in $\Qco(X)$ is pure. If the two notions of purity were the same then it is not hard to see (\cite[Corollary 3.12]{EEO}) that the categorical flats in $\Qco(X)$ are precisely the usual flats in $\Qco(X)$ (that is, flatness in terms of the stalks). But this is not the case; for instance for projective spaces, it was shown in \cite[Corollary 4.6]{ES} that the only categorical flat sheaf is the zero sheaf in this case.

The goal of this paper is to define the pure derived category of a scheme, but using the geometric purity defined above in terms of the stalks. This leads us to work in the general setting of a closed symmetric monoidal Grothendieck category. Note that not every such category need be locally finitely presentable. However, every Grothendieck category is at least locally $\lambda$-presentable for some regular cardinal $\lambda$.   
There is a nice treatment of such categories in \cite{AR}. In particular Adameck and Rosicky showed that there is a nice extension of the usual purity theory to this setting, yielding a $\lambda$-purity theory. And it is an easy, but relevant for our purposes, observation that every $\lambda$-pure monomorphism is in fact a $\otimes$-pure monomorphism whenever we are in a closed symmetric monoidal Grothendieck category. This allows to take advantage of some results and techniques on $\lambda$-Purity Theory and apply them to $\otimes$-Purity. Thus we get the following result:

\medskip\par\noindent
{\bf Theorem A:} Let $\C$ be a closed symmetric monoidal Grothendieck category and $\Ch(\C)$ the associated category of chain complexes. Then there is a cofibrantly generated model category structure on $\Ch(\C)$ whose trivial objects are the $\otimes$-pure acyclic complexes; that is, complexes $X$ for which $X \otimes S$ is exact for all $S \in \C$. The model structure is exact (i.e. abelian) with respect to the exact category $\Ch(\C)_{\otimes}$ of chain complexes along with the proper class of degreewise $\otimes$-pure exact sequences. In fact, the model structure is injective in the sense that every complex is cofibrant and the trivially fibrant complexes are the injective objects of $\Ch(\C)_{\otimes}$, which are precisely the contractible complexes with $\otimes$-pure injective components. 
We call this model structure the $\otimes$-{\it pure injective model structure} on $\Ch(\C)$ and its corresponding homotopy category is the $\otimes$-{\it pure derived category}, denoted $\mathcal D_{\otimes{\mbox{-}}{\rm pur}}(\C)$.

\medskip\par\noindent 
This model structure is a $\otimes$-pure analog to the usual injective model structure on $\Ch(\C)$ whose fibrant objects are the DG-injective complexes. Indeed the fibrant objects,  which are described in Section~\ref{subsec-pure exact struc},  are defined exactly like the DG-injective complexes but with respect to the exact structure $\Ch(\C)_{\otimes}$ instead of the full abelian exact structure on $\Ch(\C)$. In order to construct the model structure we show that these fibrant objects are the right half of an injective cotorsion pair in the exact category $\Ch(\C)_{\otimes}$. It follows from Hovey's correspondence~\cite{hovey2}
that we get the described model structure on $\Ch(\C)$.

As a particular instance of the previous theorem, we get by applying Proposition~\ref{prop-pure}, the following application to $\Qco(X)$:

\medskip\par\noindent
{\bf Corollary:} Let $X$ be a quasi-separated scheme. Let $\mathcal E$ be the exact structure coming from the stalkwise-purity in $\Qco(X)$, and let us consider the category of unbounded complexes $\Ch(\Qco(X))$. Then with respect to the induced degreewise exact structure from $\mathcal E$, there is an exact and injective model category structure on $\Ch(\Qco(X))$. The corresponding homotopy category is the \emph{stalkwise-pure derived category} (or \emph{geometric pure derived category}), which we denote $\mathcal D_{{\rm stk}{\mbox{-}}{\rm pure}}(\Qco(X))$.

\medskip\par\noindent

Having two different notions of purity in a general closed symmetric monoidal Grothendieck category, the $\lambda$-purity and the $\otimes$-purity,  it is natural to ask what relationship there is between them. In Proposition~\ref{prop-adjunction} we show that there is a canonical functor from the $\lambda$-pure derived category $\mathcal{D}_{\lambda{\mbox{-}}{\rm pur}}(\C)$ to the $\otimes$-pure derived category $\mathcal{D}_{\otimes{\mbox{-}}{\rm pur}}(\C)$, which admits a right adjoint.

\medskip\par\noindent

The second part of this paper deals with an alternative approach to defining the pure derived category of a scheme. In \cite{MS} Murfet and Salarian define what they call the {\it pure derived category of flat sheaves} for a semi-separated noetherian scheme, as the Verdier quotient of the homotopy category of flat sheaves $\mathbf{K}(\mathrm{Flat}(X))$ with the localising subcategory $\mathbf{K}_{{\rm pac}}(\mathrm{Flat}(X))$ of the pure acyclic complexes of sheaves (that is, acyclic complexes of flat sheaves with flat cycles). Flat modules are intimately related with locally finitely presented categories due to Crawley-Boevey's Representation Theorem \cite{CB}. This establishes that every locally finitely presented additive category $\mathcal A$ is equivalent to the full subcategory $\mathrm{Flat}(A)$ of $\operatorname{Mod-}\!\! A$ of unitary flat right $A$-modules, where $A$ is the functor ring of $\mathcal A$ and the equivalence gives a 1-1 correspondence between pure  exact sequences in $\mathcal A$ and exact sequences in $\mathrm{Flat}(A)$. This equivalence lifts to the level of model structures as well and in particular to the derived categories, so we get the following:

\medskip\par\noindent
{\bf Theorem B:} Let $\mathcal A$ be a locally finitely presented additive category and let $\mathrm{Flat}(A)$ be its equivalent full subcategory of flat modules in $\operatorname{Mod-}\!\! A$. Then $\mathcal D_{\rm pur}(\mathcal A)$ is equivalent to $\mathcal D(\mathrm{Flat}(A))$, the homotopy category of the injective exact model category structure on $\Ch(\mathrm{Flat}(A))$.

\medskip\par\noindent
We do not known whether $\Qco(X)$ with the stalkwise-purity exact structure is equivalent to $\mathrm{Flat}(B)$ (for some ring or scheme) unless $X$ is affine. But we are able to extend Murfet and Salarian's definition to any scheme and observe that their pure derived category of flat sheaves is precisely the usual derived category of flat sheaves. Furthermore we get the derived category of flat sheaves as the homotopy category of a Quillen model category structure on $\Ch({\rm Flat}(X))$:

\medskip\par\noindent
{\bf Theorem C:} Let $X$ be any scheme, and ${\rm Flat}(X)$ the category of quasi-coherent flat sheaves. There is an injective exact model structure on $\Ch({\rm Flat}(X))$.  So every object is cofibrant and the fibrant objects are dg-cotorsion complexes which are flat on each degree. The trivial objects are those in $\Ch_{ac}({\rm Flat}(X)) = \widetilde{\mathcal{F}}$, the class of acyclic complexes with flat cycles. The corresponding homotopy category is the derived category of flat sheaves, $\mathcal D(\mathrm{Flat}(X))$.

\section{Purity}\label{sec-purity}

In this section we gather some known facts regarding purity which will be used ahead.

\subsection{Purity in locally presentable categories}
Let $\lambda$ be a regular cardinal.
\begin{definition}\cite[2.1, page 68]{AR}
A category $\C$ is called $\lambda$-accessible if $\C$ has $\lambda$-directed colimits and $\C$ has a set of $\lambda$-presentable objects such that every object in $\C$ is a $\lambda$-directed colimit of objects from that set.
\end{definition}
\begin{definition}\cite[1.17, pg 21]{AR}
$\C$ is called locally $\lambda$-presentable if  it is cocomplete and $\lambda$-accessible.
\end{definition}
\begin{definition}\cite[2.27, page 85]{AR}\label{def-pure}
Let $f:A \rightarrow B$ be a morphism in $\C$. It is said to be $\lambda$-pure if for any commutative diagram
$$\xymatrix{A'\ar[r]^{f'}\ar[d]^u & B' \ar[d]^v\\
A\ar[r]^f & B}$$
where $A',B'$ are $\lambda$-presentable, there is a morphism $g: B' \rightarrow A$ such that $u = g \circ f'$.
\end{definition}
\begin{proposition}\cite[2.29, page 86]{AR}
Every $\lambda$-pure morphism in a $\lambda$-accessible category is a monomorphism.
\end{proposition}
\begin{proposition}\label{lp}\cite[2.30, page 86]{AR}
Let $\C$ be a locally $\lambda$-presentable category. Then a morphism is a $\lambda$-pure monomorphism if and only if it is a $\lambda$-directed colimit of split monomorphisms.
\end{proposition}
\begin{theorem}\label{sp}\cite[2.33]{AR}
(Every $\lambda$-accessible  category has enough $\lambda$-pure subobjects.) Let $\C$ be a $\lambda$-accessible category. There exist arbitrary large regular cardinals $\gamma \rhd \lambda$ such that  every $\gamma$-presentable subobject $A$ of $B$ in $\C$ is contained in a $\lambda$-pure subobject $\overline{A}$ of $B$, where $\overline{A}$ is $\gamma$-presentable.
\end{theorem}
\subsection{Purity in a locally presentable monoidal category}
Let $\C$ be a locally $\lambda$-presentable and symmetric monoidal category. Let $\mathcal{G}$ be a generating set of locally $\lambda$-presentable objects. Suppose that $\C$ has images and $\otimes$ preserves $\lambda$-colimits (for in case that $\otimes$ is not \emph{closed}).

\begin{definition}\cite{Fox}\label{def-fox}
A monomorphism $f:X \rightarrow Y$ is called $\otimes$-pure if $f\otimes Z$ is monomorphism for all $Z \in \C$.
\end{definition}
\begin{remark}\label{rem}
By Proposition \ref{lp} and the fact that $\otimes$ preserves $\lambda$-colimits, if a morphism is $\lambda$-pure then it is $\otimes$-pure.
\end{remark}

\subsection{Purity in a closed symmetric monoidal Grothendieck category}
Let $\C$ be a closed symmetric monoidal Grothendieck category. In \cite[Proposition 3]{Beke} Beke showed that every Grothendieck category is locally presentable, i.e., there is a regular cardinal $\lambda$ for which $\C$ is locally $\lambda$-presentable category. And in fact it can be shown that every object is presentable. 

\begin{proposition}\label{ps}
Let $\{P_i;\ \psi_{ij}:P_i \rightarrow P_j\}_I$ be a $\lambda$-directed system in  $\C$. Then the canonical morphism $\bigoplus_I P_i \rightarrow  {\rm colim} P_i \rightarrow 0$ is  $\lambda$-pure epic. So it is also $\otimes$-pure epic.
\end{proposition}
\begin{proof}
For any $\rho: i \rightarrow j$, $s(\rho)=i$ and $t(\rho)=j$. For each $i\in I$, let us denote by $\iota_i:P_i \rightarrow \bigoplus_I P_i$ and $\pi_i : \bigoplus P_i \rightarrow P_i$ the canonical injection and projection maps respectively. Consider $l_{\rho}:=\iota_{t(\rho)} \circ \psi_{ij}-\iota_{s(\rho)}:  P_{s(\rho)} \rightarrow \bigoplus_I P_i$, which is monic for all morphism $\rho$ in $I$, (actually it splits). So, we have the induced morphism $(l_\rho)_{\rho}: \bigoplus_{\rho}P_{s(\rho)} \rightarrow \bigoplus_I P_i$. We know that ${\rm colim} P_i = \coker (l_{\rho})_{\rho}$. That is, there is an exact sequence
$$\xymatrix{0 \ar[r]& \sum_{\rho}\im l_{\rho} \ar[r]& \bigoplus_I P_i \ar[r]^g & {\rm colim} P_i \ar[r] & 0}.$$
Then $\alpha_i:= g \circ \iota_i$ is the family of morphisms $\alpha_i : P_i \rightarrow {\rm colim} P_i$ with $\alpha_j \circ \psi_{ij}= \alpha_i$ for each $\rho:i \to j$.
Let $f: H\rightarrow {\rm colim} P_i $ be a morphism where $H$ is $\lambda$-presentable. Then $f$ factors through  $\alpha_i$ for some $i$, that is, there is a morphism $f': H \rightarrow P_i$ such that $\alpha_i \circ f' =f$. But $\alpha_i = g \circ \iota_i$,  $g\circ \iota_i \circ f' =f$. That is, that exact sequence is $\Hom(H,-)$-exact for each $\lambda$-presentable object $H$, which means that it is $\lambda$-pure exact.
\end{proof}
\subsection{Stalkwise-purity in $\Qco(X)$}
Let $X$ be a scheme with associated structure sheaf $\mathcal O_X$. The category $\Qco(X)$ is a closed symmetric monoidal Grothendieck category, with the closed structure coming from the \emph{coherator} functor $Q$ applied to the usual sheafhom. Therefore we can define $\otimes$-pure monomorphisms as in Definition~\ref{def-fox}. But also we can give a local notion of purity in term of the stalks: a monomorphism $f:\scrF\to \scrG$ in $\Qco(X)$ is called stalk-wise pure if for each $x \in X$ the induced morphism $f_x:\scrF_x\to \scrG_x$ on the stalks is pure in $\mathcal O_{X,x}\Mod$. The next proposition, adapted from\cite[Propositions 3.3 and 3.4]{EEO}, relates the two notions. The careful reader will notice that the definition of $\otimes$-pure used here is slightly different than the one used in~\cite{EEO}. But they agree when $X$ is quasi-separated. See~\cite[Remark~3.5]{EEO}.
\begin{proposition}\label{prop-pure}
Let $X$ be a quasi-separated scheme, and $f:\scrF\to \scrG$ a monomorphism in $\Qco(X)$. The following assertions are equivalent:
\begin{enumerate}
\item $f$ is $\otimes$-pure.
\item There exists an open affine covering $\mathscr{U}=\{U_i\}_{i\in I}$ of $X$ such that $f_{U_i}$ is pure in $\mathcal O_X(U_i)\Mod$.
\item $f$ is stalk-wise pure.

\end{enumerate}
\end{proposition}
\begin{proof}
$1.\Rightarrow 2.$ Let  $U$ be an affine open subset of $\mathscr{U}$ and $i:U\hookrightarrow X$ be the open immersion. And let $M\in \O_X(U)\Mod$. Since $X$ is quasi-separated, then $i_*(\widetilde{M})$ is a quasi-coherent $\O_X$-module. Therefore $$0\to i_*(\widetilde{M})\otimes \scrF\to i_*(\widetilde{M})\otimes \scrG  $$is exact. But then $$0\to (i_*(\widetilde{M})\otimes \scrF)(U)\to (i_*(\widetilde{M})\otimes \scrG)(U)  $$ is exact in $\O_X(U)\Mod$, that is $$0\to i_*(\widetilde{M})(U)\otimes \scrF(U)\to i_*(\widetilde{M})(U)\otimes \scrG(U)  $$is exact. Since, for each $\O_X(U)$-module $A$, $i_*(\widetilde{A})(U)=A$, we get that $0\to M\otimes \scrF(U)\to M\otimes \scrG(U)$ is exact. Thus $0\to \scrF(U)\to \scrG(U)$ is pure.

\medskip\par\noindent
$2.\Rightarrow 3.$ Let $x\in X$. Then there exists $U_i\in \mathscr{U}$ such that $x\in U_i=Spec(A_i)$, for some ring $A_i$. But then the claim follows by observing that $\scrF_x=(\widetilde{\scrF(U_i)})_x\cong  \widetilde{\scrF(U_i)_x}$ and noticing that if $0\to M\to N$ is pure exact in $A_i\Mod$, then $0\to M_x\to N_x$ is pure exact in $(A_i)_x\Mod$.

\medskip\par\noindent
$3. \Rightarrow 1.$ Let  $\scrF\stackrel{f}{\to} \scrG$ be a monomorphism in $\Qco(X)$ (so, for each $x\in X$, $0\to \scrF_x\stackrel{f_x}{\to} \scrG_x$ is exact in $\O_{X,x}\Mod$). Given $\scrM\in \Qco(X)$, the induced $ \scrM\otimes \scrF\stackrel{id\otimes f}{\longrightarrow}\scrM\otimes \scrG$ will be a monomorphism if, and only if, for each $x\in X$ the morphism of $\O_{X,x}$-modules $(\scrM\otimes \scrF)_x\stackrel{(id\otimes f)_x}{\longrightarrow}(\scrM\otimes \scrG)_x$ is such. But, for each $x\in X$, and $\scrA\in \Qco(X)$, $(\scrM\otimes \scrA)_x\cong \scrM_x\otimes \scrA_x$. So, by $3.$ we follow that $ \scrM\otimes \scrF\stackrel{id\otimes f}{\longrightarrow}\scrM\otimes \scrG$ is a monomorphism. Therefore $\scrF\stackrel{f}{\to} \scrG$ is $\otimes$-pure.
\end{proof}

\section{The pure-injective model structure}

Our main goal here is to prove Theorem A of the introduction and its Corollary. 

\subsection{Exact categories of Grothendieck type}
Let $\C$ be an exact category.

\begin{definition}\cite[Definition 3.2]{J}
Let $\alpha$ be an ordinal number, and let \newline $ (X_{\beta}, f_{\beta\beta'})_{\beta<\beta' < \alpha}$ be a direct system indexed by $\alpha$ in $\C$. Such a system is called a \emph{$\alpha$-sequence} if for each limit ordinal $\beta < \alpha$, the object $X_{\beta}$ together with the morphisms $f_{\mu\beta}$, $\mu < \beta$, is a colimit of the direct subsystem $(X_{\mu}, f_{\mu,\mu '})_{\mu< \mu' < \beta }$.
\end{definition}

\begin{definition}\cite[Definition 3.3]{J}
If $\C$ is an exact category,  $\kappa$ is  a cardinal number and $\mathcal{D}$ is a class of morphisms of $\C$, then an object $X \in \C$ is called \emph{$\kappa$-small relative to $\mathcal{D}$} if for every infinite regular cardinal $\alpha \geq  \kappa$ and every $\alpha$-sequence $(E_{\beta}, f_{\beta \beta'})_{\beta<\beta' < \alpha}$ in $\C$ such that $f_{\beta,\beta+1} \in \mathcal{D}$ for all $\beta+1 < \alpha$, the canonical map of sets
$$\varinjlim_{\beta< \alpha} \Hom_{\C}(X, E_{\beta}) \rightarrow \Hom(X,\varinjlim_{\beta < \alpha} E_{\beta})$$
is an isomorphism. The object $X$ is \emph{small relative to $\mathcal{D}$} if it is $\kappa$-small relative to $\mathcal{D}$ for some cardinal $\kappa$.
\end{definition}

\begin{definition}\cite[Definition 3.4]{J}\label{def-efficient}
An exact category $\C$ is called \emph{efficient} if
\begin{enumerate}
\item $\C$ is weakly idempotent complete. That is, every section
$s: X
\rightarrow Y$ in $\C$ has a cokernel or, equivalently, every retraction $r:Y\rightarrow Z$ in $\C$ has a kernel.
\item Arbitrary transfinite compositions of inflations exist and are
themselves inflations.
\item Every object of $\C$ is small relative to the class of all inflations.
\item $\C$ admits a generator. That is, there is an object $G \in \C$ such that every $X \in \C$ admits a deflation $G^{(I)} \rightarrow X \rightarrow 0$.
\end{enumerate}
\end{definition}
\begin{definition}\cite[Definition 3.11]{J}\label{def-Grot type}
An exact category $\C$ is said to be  of \emph{Grothendieck type} if it is efficient and it is deconstructible in itself, i.e, there is a set of objects $S \subset \C$ such that $\C = \Filt (S)$.
\end{definition}
 
Note that whenever $\C$ is an exact category, then the chain complex category $\bfC(\C)$ is also an exact category whose conflations are pointwise conflations in $\C$. Unless explicitly stated otherwise, $\bfC(\C)$ will always denote this exact structure. We get a notion of an exact complex (or acyclic complex) and we let $\bfC _{ac}(\C)$ denote the class of all exact complexes.

\begin{lemma}\label{Je}\cite[Lemma 7.10]{J}
Let $\C$ be an exact category of Grothendieck type such that $\bfC_{ac}(\C)$ is deconstructible in $\bfC(\C)$. If $(\mcF, \mcB)$ is a complete hereditary cotorsion pair in $\C$, then $(\F, \F^{\perp})$ is a complete (and hereditary) cotorsion pair in $\bfC (\C)$. 
\end{lemma} 

\noindent In the above lemma, we have used that $\mcF$ is extension closed and so inherits an exact structure from $\C$, and $\F = \bfC_{ac}(\mathcal{F})$.

\subsection{Pure exact structure}\label{subsec-pure exact struc}
Let $\C$ be a closed symmetric monoidal Grothendieck category. In Section~\ref{sec-purity} we noted that $\C$ is locally $\lambda$-presentable for some regular cardinal $\lambda$, and that there are two generally different notions of purity. Let $\mathcal{P}$ denote the proper class of $\lambda$-pure short exact sequences in $\C$ and $\mathcal{P}_{\otimes}$ denote the proper class of $\otimes$-pure short exact sequences in $\C$. By Remark~\ref{rem} we have the containment $\mathcal{P} \subseteq \mathcal{P}_{\otimes}$. Our main interest in this section will be the $\otimes$-pure exact structure. \textbf{\emph{So throughout the rest of this section, when we say pure exact we will always mean $\otimes$-pure exact}}, unless explicitly stated otherwise. We will denote by $\bfC(C)_{\otimes}$ the exact structure consisting of $\bfC(C)$ along with the componentwise pure exact sequences. 
\begin{lemma}\label{Gtype}
$\C$ with the pure exact structure is an exact category of Grothendieck type.
\end{lemma}
\begin{proof}
It is routine to check that $\C$ along with the pure exact sequences form an exact category. So we check Definitions~\ref{def-efficient} and~\ref{def-Grot type}.  
First, we must prove that it is efficient. $(i)$ is clear since $\C$ is abelian and $(ii)$ is also clear since the tensor product preserves any colimit. Any object $X \in \C$ is $\kappa$-presentable for some cardinal $\kappa$, so $(iii)$ easily follows. Now $(iv)$ follows from Proposition~\ref{ps}. In detail since $\C$ is locally $\lambda$-presentable we have a set of $\lambda$-presentable objects for which each $X \in \C$ is the colimit of a $\lambda$-directed system $\{P_i;\ \psi_{ij}:P_i \rightarrow P_j\}_I$ with each $P_i$ in that set. Then by Proposition~\ref{ps} the canonical morphism $\bigoplus_I P_i \rightarrow  {\rm colim} P_i \rightarrow 0$ is a pure epimorphism as required.
So we conclude that $\C$ with the pure exact structure is efficient.  Finally by Theorem~\ref{sp}, Remark~\ref{rem}, and the fact that if $A\leq A' \leq B$ is such that $A\leq B$ and $A'/A \leq B/A$ are pure-monic in $\C$ then $A' \leq B $ is also pure-monic, we infer that there is a regular cardinal $\gamma$ such that $\C= \Filt (\C^{\leq \gamma})$. Here $\C^{\leq \gamma}$ is the class of all $\gamma$-presentable objects in $\C$ and the filtration is built in $\C$ with the pure-exact structure.
\end{proof}

A complex $C$ in $\bfC(C)$ is called \emph{$\otimes$-acyclic} if it is acyclic in $\bfC(\C)_{\otimes}$, the exact category of chain complexes with the pointwise pure exact structure. This means
each sequence $0 \rightarrow Z_nC \rightarrow C_n \rightarrow Z_{n-1}C \rightarrow 0$ is pure, or equivalently, $C \otimes S$ is exact for all $S \in \C$. We shall denote by $\bfC_{\otimes\textrm{-}ac}(\C)$ the class of all $\otimes$-acyclic complexes. Our aim is to construct the relative derived category of $\C$ with respect to the $\otimes$-pure proper class, that is, whose trivial objects are the $\otimes$-acyclic complexes. To achieve this aim, we will use Hovey's correspondence between cotorsion pairs and model category structures~\cite{hovey2}. We note that when the underlying category is abelian an exact structure on the category is the same thing as a proper class~\cite[Appendix~B]{G}. So the language of abelian model structures from~\cite{hovey2} and the language of exact model structures from~\cite{G4} and~\cite{J} are the same thing when the underlying category is abelian. 

Let $\Pinj$ denote the class of objects in $\C$ having the injective property with respect to the proper class $\mathcal{P}_{\otimes}$, the $\otimes$-pure short exact sequences in $\C$. We will call an object in $\Pinj$ a \emph{pure-injective}. We recall the following proposition from \cite{J}.

\begin{proposition}\cite[Corollary 5.9]{J}\label{prop-inj}
Let $(\C,\mathcal E)$ be an exact category of Grothendieck type and $\Inj$ the class of injective objects with respect to $\mathcal E$. Then $(\C, \Inj)$ is a functorially complete cotorsion pair in $\C$.
\end{proposition}

From this proposition and Lemma \ref{Gtype}, we get that $(\C, \Pinj)$ is a  hereditary complete cotorsion pair in $\C$ with the pure exact structure. In particular every object in $\C$ can be purely embedded in a pure-injective object.

We now define the following classes in $\bfC(\C)$, which will turn out to be the fibrant and trivially fibrant objects in our model structure for the $\otimes$-pure derived category:
$$ \dg\Pinj=\{L\in\bfC(\C):\ L_n\in \Pinj \textrm{\ and\ each\ map\ } E\to L\ \textrm{is\ homotopic\ to}\ 0, \forall E\in \bfC_{\otimes\textrm{-}ac}(\C) \} $$ and 
$$\widetilde{\Pinj}=\{T\in \bfC_{\otimes\textrm{-}ac}(\C):\ Z_nT\in \Pinj\}. $$
One can check that $\widetilde{\Pinj}$ is the class of injective objects in the exact category $\bfC(\C)_{\otimes}$ of chain complexes with the pointwise pure-exact structure. Moreover, they are precisely the contractible complexes with pure-injective components. 
We want to apply \cite[Theorem 2.2]{hovey2} to the pairs $(\bfC_{\otimes\textrm{-}ac}(\C), \dg \Pinj)$ and $(\bfC(\C), \widetilde{\Pinj})$. So we have to show that these two pairs are complete cotorsion pairs in $\bfC(\C)_{\otimes}$.

First, we will prove that $\bfC_{\otimes\textrm{-}ac}(\C)$ is deconstructible. We will start with the following lemma.
\begin{lemma}\label{prev}
Let $X$ be $\otimes$-acyclic and $X'$ be a subcomplex of $X$. Assume that 
\begin{enumerate}
\item $X'$ is acyclic.
\item $Z_n X' \subseteq Z_n X$ is $\otimes$-pure, for each $n \in \mathbb{Z}$.
\end{enumerate}
Then $X'$ is $\otimes$-acyclic and $X'_n\subseteq X_n$ is $\otimes$-pure, for each $n\in \mathbb{Z}$.
\end{lemma}
\begin{proof}
By hypothesis, we have the commutative diagram below with the top row exact, the bottom row pure exact, and the outer vertical arrows pure monomorphisms. 
$$\begin{CD}
 0 @>>> Z_nX' @>>> X_n' @>>> Z_{n-1}X'  @>>> 0 \\
@. @VVV @VVV @VVV @. \\
 0 @>>> Z_nX @>>> X_n @>>> Z_{n-1}X  @>>> 0 \\
\end{CD}$$ Since the composite $Z_nX' \subseteq Z_nX \subseteq X_n$ is pure, we get that the composite $Z_nX' \subseteq X'_n \subseteq X_n$ is also pure. It follows immediately that $Z_nX' \subseteq X'_n$ is pure. So the top row is pure exact and now the snake lemma can be used to show that the middle vertical arrow is also a pure monomorphism. 
\end{proof}
\begin{proposition}\label{filtr}
There is a regular cardinal $\gamma$ such that every $\otimes$-acyclic complex $X$ has a $\bfC_{\otimes\textrm{-}ac}(\C)^{\leq \gamma}$-filtration. That is, $\bfC_{\otimes\textrm{-}ac}(\C)= \Filt (\bfC_{\otimes\textrm{-}ac}(\C)^{\leq \gamma})$, where $\bfC_{\otimes\textrm{-}ac}(\C)^{\leq \gamma}$ is the class of $\gamma$-presentable $\otimes$-acyclic complexes, and the monomorphisms in the filtration are with respect to the degreewise pure exact structure.
\end{proposition}
\begin{proof}
The class $\bfC_{\otimes\textrm{-}ac}(\C)$ is closed under direct limits, so it suffices to show that there is a regular cardinal $\gamma$ satisfying that: given $A\subseteq X\neq 0$ where $X\in \bfC_{\otimes\textrm{-}ac}(\C)$ and $A$ is $\gamma$-presentable, there exists a $\gamma$-presentable $X'\neq 0$ such that $A\subseteq X'\subseteq X$, and $X'\in \bfC_{\otimes\textrm{-}ac}(\C)$, and $X'_n\subseteq X_n$ is $\otimes$-pure for each $n\in \mathbb{Z}$. Once we show this, a standard argument utilizing properties of the $\otimes$-purity will allow for the construction of the desired filtration of $X$.

Since $\C$ is Grothendieck, it is locally $\lambda$-presentable and so $\bfC(\C)$ is also locally $\lambda$-presentable.
Let $0\neq X\in \bfC_{\otimes\textrm{-}ac}(\C)$. By Theorem \ref{sp}, there is a regular cardinal $\gamma$ such that every $\gamma$-presentable subcomplex $A\subseteq X$ can be embedded in a $\gamma$-presentable subcomplex $X'\subseteq X$ which is a $\lambda$-pure embedding. 
According to Lemma \ref{prev} we just need to check that $X'$ is acyclic and that $Z_n X' \subseteq Z_n X$ is $\otimes$-pure for all $n \in \Z$. Now for any $\lambda$-presentable $L \in \C$ we have that $S^n(L)$ is a $\lambda$-presentable complex. Therefore, applying $\Hom_{\Ch(\C)}(S^n(L),-)$ to $0 \xrightarrow{} X' \xrightarrow{} X \xrightarrow{} X/X' \xrightarrow{} 0$, yields a short exact sequence 
$$0 \xrightarrow{} \Hom_{\Ch(\C)}(S^n(L),X') \xrightarrow{} \Hom_{\Ch(\C)}(S^n(L),X) \xrightarrow{} \Hom_{\Ch(\C)}(S^n(L),X/X') \xrightarrow{} 0.$$
But under the canonical isomorphism $\Hom_{\C}(L,Z_n Y)\cong \Hom_{\Ch(\C)}(S^n(L),Y)$
this gives us a short exact sequence $0 \xrightarrow{} \Hom_{\C}(L,Z_nX') \xrightarrow{} \Hom_{\C}(L,Z_nX) \xrightarrow{} \Hom_{\C}(L,Z_n(X/X')) \xrightarrow{} 0$. Since $\C$ is locally $\lambda$-presentable, there it has a generating set consisting of $\lambda$-presented objects and so it follows that $0 \xrightarrow{} Z_nX' \xrightarrow{} Z_nX \xrightarrow{} Z_n(X/X') \xrightarrow{} 0$ is a short exact sequence. In fact we have just shown that this is a $\lambda$-pure exact sequence in $\C$. So it is also $\otimes$-pure exact. It now only remains to show that $X'$ is itself exact. For this, we apply the snake lemma to 
$$\begin{CD}
 0 @>>> Z_nX' @>>> Z_nX @>>> Z_n(X/X')  @>>> 0 \\
@. @VVV @VVV @VVV @. \\
 0 @>>> X'_n @>>> X_n @>>> (X/X')_n  @>>> 0 \\
\end{CD}$$ 
to conclude we have a short exact sequence 
$0 \xrightarrow{} B_{n-1}X' \xrightarrow{} B_{n-1}X \xrightarrow{} B_{n-1}(X/X') \xrightarrow{} 0$ for all $n$.
We then turn around and apply the snake lemma to 
$$\begin{CD}
 0 @>>> B_nX' @>>> B_nX @>>> B_n(X/X')  @>>> 0 \\
@. @VVV @VVV @VVV @. \\
 0 @>>> Z_nX' @>>> Z_nX @>>> Z_n(X/X')  @>>> 0 \\
\end{CD}$$ 
and use that $B_nX = Z_nX$ to conclude that $B_nX = Z_nX$ (and $B_n(X/X') = Z_n(X/X')$). 
\end{proof}

\begin{corollary}\label{cor-DG-pure-injective cot pair}
The pair $(\bfC_{\otimes\textrm{-}ac}(\C),\dg \Pinj)$ is a complete (and hereditary) cotorsion pair in  $\bfC(\C)_{\otimes}$.

\end{corollary}
\begin{proof}
As noted after Proposition~\ref{prop-inj}, we have that $(\C, \Pinj)$ is a complete hereditary cotorsion pair in $\C$ with the pure exact structure.
So by Lemma \ref{Je} and Proposition \ref{filtr}, we infer that $(\bfC_{\otimes\textrm{-}ac}(\C), \bfC_{\otimes\textrm{-}ac}(\C)^{\perp})$ is a complete (and hereditary) cotorsion pair in $\bfC(\C)_{\otimes}$. So it only remains to show that $\bfC_{\otimes\textrm{-}ac}(\C)^{\perp}$ coincides with the class $\dg \Pinj$. By definition of $\dg \Pinj$, it is clear $\dg \Pinj \subseteq \bfC_{\otimes\textrm{-}ac}(\C)^{\perp}$. Now let $X \in \bfC_{\otimes\textrm{-}ac}(\C)^{\perp}$. It is enough to show that each $X_i$ is pure-injective. Let $0 \rightarrow X_i \rightarrow A \rightarrow B \rightarrow 0$ be pure-exact sequence in $\C$. Then we get a short exact sequence of complexes $0 \rightarrow X \rightarrow \overline{A} \rightarrow D^i(B) \rightarrow 0$ in $\bfC(\C)$ by taking the pushout of $X_i\to X_{i-1}$ and $X_i\to A$ and where $\overline{A}_i=A$. Since pure-monomorphisms are closed by forming pushouts the sequence is degreewise pure-exact. By assumption the sequence splits and so, in particular, it splits on each degree. Hence $X_i$ is pure-injective.
\end{proof}

\begin{proposition}
The pair $(\bfC(\C), \widetilde{\Pinj})$ is a complete (and hereditary) cotorsion pair in $\bfC(\C)_{\otimes}$. Moreover, $\widetilde{\Pinj}=\dg \Pinj\cap \bfC_{\otimes\textrm{-}ac}(\C)$.
\end{proposition}

\begin{proof}
It can be easily observed that $\bfC(\C)_{\otimes}$, the exact category of chain complexes with the degreewise pure-exact structure, is of Gro\-then\-dieck type. Indeed, $\bfC(\C)$ is a Grothendieck category and any $\lambda$-pure subobject gives us a degreewise $\lambda$-pure monomorphism. So it is a degreewise $\otimes$-pure monomorphism as well.  Note that, colimits in $\bfC(\C)$ are computed pointwise.  So we again can apply Proposition~\ref{sp} to argue that $\bfC(\C)_{\otimes}$ is deconstructible in itself. Then by Proposition~\ref{prop-inj} we get that $(\bfC(\C), \Inj)$ is a  complete cotorsion pair in $\bfC(\C)_{\otimes}$. But here $\Inj=\widetilde{\Pinj}$.

\medskip\par\noindent
Since $\widetilde{\Pinj}$ consists of contractible complexes of pure-injectives, $\widetilde{\Pinj} \subseteq \dg \Pinj\cap \bfC_{\otimes\textrm{-}ac}(\C)$.  For the converse, let $X \in \dg \Pinj\cap \bfC_{\otimes\textrm{-}ac}(\C)$. By assumption, the identity map $X\to X$ is homotopic to zero, so $X$ is a contractible complex of pure-injectives. So $X \in \widetilde{\Pinj}$.
\end{proof}

Now note that $\bfC_{\otimes\textrm{-}ac}(\C)$ is thick in the exact category $\bfC(\C)_{\otimes}$. So we have now proved {\bf Theorem A} of the introduction.

\begin{remark}
If $\lambda'\geq \lambda$ are regular cardinals, then any $\lambda$-presentable object is also $\lambda'$-presentable. So by Definition~\ref{def-pure} we see that $\lambda'$-pure implies $\lambda$-pure. (Warning! There is a misprint on the bottom of page~85 of~\cite{AR}.) We also see from Proposition~\ref{lp} that if $\C$ is a locally $\lambda$-presentable additive category, then the $\lambda$-pure exact structure is the smallest exact structure on $\C$ that is closed under $\lambda$-directed colimits. We conclude that if $\C$ is a closed symmetric monoidal Grothendieck category and locally $\lambda$-presentable, then we have containments of exact structures: $\mathcal{P}_{\lambda'} \subseteq \mathcal{P}_{\lambda} \subseteq \mathcal{P}_{\otimes}$ whenever $\lambda' \geq \lambda$. 
\end{remark}

\section{Relationship  between the two pure derived categories}

Suppose that $\C$ is a closed symmetric monoidal Grothendieck category. In this section we get an adjunction between the two derived categories obtained from the proper class $\mathcal{P}$ of the $\lambda$-pure short exact sequences and the proper class $\mathcal{P}_{\otimes}$ of the $\otimes$-pure short exact sequences. Recall that  $\mathcal P \subseteq \mathcal{P}_{\otimes}$. By \cite{G}, we have  the $\lambda$-pure derived category $\mathcal{D}_{\lambda{\mbox{-}}{\rm pur}}(\C)$ and the $\lambda$-pure projective model structure on $\bfC(\C)_{\mathcal{P}}$ whose trivial objects are the $\lambda$-pure exact complexes. This model structure corresponds to Hovey pairs in $\bfC(\C)_{\mathcal{P}}$ that we denote by $(\dg \Pproj, \bfC_{\lambda\textrm{-}ac}(\C))$ and $(\widetilde{\Proj}, \bfC(\C))$. In particular, $\bfC_{\lambda\textrm{-}ac}(\C)$ denotes the class of $\lambda$-pure exact complexes.

 From the previous section, we have the Hovey pairs $(\bfC_{\otimes\textrm{-}ac}(\C), \dg \Pinj)$ and $(\bfC(\C), \widetilde{\Pinj})$ on $\bfC(\C)_{\otimes}$. So the derived category $\mathcal{D}_{\otimes{\mbox{-}}{\rm pur}}(\C)$ has an injective model structure whose trivial objects are the $\otimes$-acyclic complexes, while $\mathcal{D}_{\lambda{\mbox{-}}{\rm pur}}(\C)$ has a projective model structure whose trivial objects are the $\lambda$-pure acyclic complexes. 
\begin{definition}
Suppose $\C$ and $\mathcal{D}$ are model categories.
\begin{enumerate}
\item We call a functor $F: \C \rightarrow \mathcal{D}$ a \emph{left Quillen functor} if $F$ is a left adjoint and preserves cofibrations and trivial cofibrations.
\item We call a functor $U: \mathcal{D} \rightarrow \C$ a \emph{right Quillen functor} if $U$ is a right adjoint and preserves fibrations and trivial fibrations.
\item Suppose $(F, U, \varphi)$ is an adjunction from $\C$ to $\mathcal{D}$. That is, $F$ is a functor $\C \rightarrow \mathcal{D}$, $U$ is a functor $\mathcal{D} \rightarrow \C$, and $\varphi$ is a natural isomorphism $\Hom(FA,B) \rightarrow \Hom(A,UB)$ expressing $U$ as a right adjoint of $F$. We call $(F, U, \varphi)$ a \emph{Quillen adjunction} if $F$ is a left Quillen functor.
\end{enumerate}
\end{definition}
\begin{lemma}\label{aQ1}\cite[Lemma 1.3.4]{hovey2}
Suppose $(F,U,\varphi): \C \rightarrow \mathcal{D}$ is an adjunction, and $\C$ and $\mathcal{D}$ are
model categories. Then $(F,U,\varphi)$ is a Quillen adjunction if and only if $U$ is a right
Quillen functor.
\end{lemma}
\begin{definition}
Suppose $\C$ and $\mathcal{D}$ are model categories.
\begin{enumerate}
\item If $F : \C \rightarrow \mathcal{D}$ is a left Quillen functor, define the total left derived functor
$LF: \Ho \C \rightarrow \Ho \mathcal{D}$ to be the composite $$\xymatrix{\Ho \C \ar[r]^{\HoQ} & \Ho \C_c \ar[r]^{\HoF} & \Ho \mathcal{D}}.$$
Given a natural transformation $\tau : F \rightarrow F' $ of left Quillen functors, define the total derived natural transformation $L_{\tau}$ to be $ \Ho \tau \circ \Ho Q$, so that $(L \tau)_X = \tau_{QX}$.
\item  If $U : \mathcal{D} \rightarrow \C$ is a right Quillen functor, define the total right derived functor $RU: \Ho \mathcal{D} \rightarrow \Ho \C$ of $U$ to be the composite $$\xymatrix{ \Ho \mathcal{D} \ar[r]^{\HoR} &  \Ho \mathcal{D}_f \ar[r]^{HoU} & \Ho \C}.$$
Given a natural transformation $\tau : U \rightarrow U'$ of right Quillen functors, define the total derived natural transformation $R \tau$ to be $\Ho \tau \circ \Ho R$, so that $R\tau_X =\tau_{RX}X$.
\end{enumerate}
\end{definition}
\begin{lemma}\label{aQ2}\cite[Lemma 1.3.10]{hovey2}
Suppose $\C$ and $\mathcal{D}$ are model categories and $(F,U,\varphi): \C \rightarrow  \mathcal{D}$
is a Quillen adjunction. Then $LF$ and $RU$ are part of an adjunction $L(F,U,\varphi) =
(LF,RU,R\varphi)$, which we call the derived adjunction.
\end{lemma}
\begin{proposition}\label{prop-adjunction}
$\id: \bfC(\C)_{\mathcal{P}} \rightarrow \bfC(C)_{\otimes}$ is a left Quillen functor. So there is a Quillen adjunction between $\mathcal{D}_{\lambda{\mbox{-}}{\rm pur}}(\C)$ and $\mathcal{D}_{\otimes{\mbox{-}}{\rm pur}}(\C)$
\end{proposition}
\begin{proof}
Clearly $\id$ is a left adjoint functor of $\id:\bfC(\C)_{\otimes} \rightarrow \bfC(C)_{\mathcal{P}}$. Also, $\id$ preserves cofibrations and trivial cofibrations. Indeed, a cofibration $f$ in $\bfC(\C)_{\mathcal{P}}$ is a degreewise $\lambda$-pure monomorphism with cokernel in $\dg \Pproj$. Such an $f$ is a cofibration in $\bfC(\C)_{\otimes}$ as well since here the cofibrations are the degreewise $\otimes$-pure monomorphisms. Also, any trivial cofibration $f$ in $\bfC(\C)_{\mathcal{P}}$ is a degreewise $\lambda$-pure monomorphism with cokernel in $\widetilde{\Proj}$, in particular, contractible. So it is a trivial cofibration in $\bfC(\C)_{\otimes}$.

By Lemma \ref{aQ1}, $\id:\bfC(\C)_{\otimes} \rightarrow \bfC(C)_{\mathcal{P}} $ is a right Quillen functor.

From Lemma \ref{aQ2}, the  total left derived functor $L(\id): \mathcal{D}_{\lambda{\mbox{-}}{\rm pur}}(\C) \rightarrow \mathcal{D}_{\otimes{\mbox{-}}{\rm pur}}(\C)$ and the total right derived functor $R(\id):\mathcal{D}_{\otimes{\mbox{-}}{\rm pur}}(\C) \rightarrow \mathcal{D}_{\lambda{\mbox{-}}{\rm pur}}(\C) $ gives us an adjunction $(L(\id),R(\id))$. By definition, $L(\id)(X)$ is its cofibrant replacement in $\bfC(\C)_{\mathcal{P}}$, that is, $L(\id)(X) \in \dg \Pproj$. Dually, $R(\id)(X)$ is its fibrant replacement in $\bfC(\C)_{\otimes}$, so $R(\id)(X) \in \dg \Pinj$.
\end{proof}

\section{The pure derived category of flat sheaves via model structures}
Let $\mathcal{A}$ be a locally finitely presentable additive category. We wish to prove the remaining two theorems from the introduction.  
We start by recalling the following representation theorem due to Crawley-Boevey (see also \cite[Chapter 16]{Prest2} for a nice exposition).
\begin{theorem}[Crawley-Boevey]
Every locally finitely presented additive category $\mathcal A$ is equivalent to the full subcategory $\mathrm{Flat}(A)$ of the category $\operatorname{Mod-}\!\! A$ of unitary right $A$-modules consisting of flat right $A$-modules where $A$ is the functor ring of $\mathcal A$ (that is, a ring with enough idempotents). This equivalence gives a 1-1 correspondence between pure exact sequences in $\mathcal A$ and exact sequences in $\mathrm{Flat}(A)$.
\end{theorem}
In other words, $\mathcal{A}$ with its pure exact structure is equivalent to $\mathrm{Flat}(A)$ with its canonical exact structure inherited from $\operatorname{Mod-}\!\! A$. In particular, the equivalence takes injective objects in $\mathcal{A}$ (pure-injectives) to injective objects in $\mathrm{Flat}(A)$ (cotorsion flat modules). Similarly it preserves projectives, taking pure-projectives in $\mathcal{A}$ (retracts of direct sums of finitely presented objects) to projective modules in $\mathrm{Flat}(A)$. Also each exact category is of Grothendieck type with the class of acyclic complexes being deconstructible. This leads to injective model structures on the associated chain complex categories with their inherited degreewise exact structures. On the other hand, each of the exact categories $\mathcal{A}$ and $\mathrm{Flat}(A)$ possesses a set of projective generators leading to projective model structures. Concentrating on the exact category $\Ch({\rm Flat}(A))$, we have the following fact from~\cite[Corollary~7.4 and~7.5]{G13}.

\begin{lemma}\label{lemma-models}
There is an injective model structure on $\Ch({\rm Flat}(A))$ in which every object is cofibrant and the fibrant objects are dg-cotorsion complexes which are flat on each degree. The trivial objects are the acyclic complexes in $\Ch({\rm Flat}(A))$. This class coincides with $\widetilde{\mathcal{F}}$, the class of exact complexes with flat cycles. On the other hand, there is a projective model structure on $\Ch({\rm Flat}(A))$ having the same class of trivial objects. Here every object is fibrant and the cofibrant objects are the complexes consisting of a projective module in each degree. 
\end{lemma}

On the other hand we learned from \cite{G} that the (usual, i.e. categorical) pure derived category of a locally finitely presented category $\mathcal A$ can be obtained as the homotopy category of both an injective and projective model category structure on the exact category $\Ch(\mathcal A)_{dw-pur}$. This denotes the exact category of chain complexes with the degreewise pure exact structure. So in view of the previous comments we have the following alternative way of defining the pure derived category of $\mathcal A$.

\begin{theorem}\label{lam}
Let $\mathcal A$ be a locally finitely presented additive category and $\mathrm{Flat}(A)$ its equivalent full subcategory of flat modules in  $\operatorname{Mod-}A$.
The (categorical) pure derived category of $\mathcal A$, $\mathcal D_{\mathrm{pur}}(\mathcal A)$, is equivalent to the derived category of the exact category $\mathrm{Flat}(A)$, $\mathcal D(\mathrm{Flat}(A))$.
\end{theorem}

\begin{proof}
Using the equivalence of Crawley-Boevey discussed above, the acyclic complexes in the exact category $\Ch(\mathcal A)_{dw-pur}$, which are the pure acyclic complexes, correspond to the class  $\widetilde{\mathcal{F}}$ of acyclic complexes in $\Ch({\rm Flat}(A))$. The injective model structure on $\Ch(\mathcal A)_{dw-pur}$ is completely determined by the injective cotorsion pair (Pure acyclic complexes, DG-pure-injectives). This corresponds to the injective cotorsion pair ($\widetilde{\mathcal{F}}$, dg-cotorsion complexes of flats) in Lemma~\ref{lemma-models}. There is a similar correspondence for the projective model structures. 
We note that by~\cite[Lemma~1]{CB} the exact structures are each weakly idempotent complete and so by~\cite[Lemma~3.1]{G13} a map is a weak equivalence in either model structure if and only if it factors as an admissible monomorphism (inflation) with trivial cokernel followed by an admissible epimorphism (deflation) with trivial kernel. From this we see that weak equivalences in  $\Ch(\mathcal A)_{dw-pur}$ correspond to weak equivalences in $\Ch({\rm Flat}(A))$. So the homotopy categories $\mathcal D_{\mathrm{pur}}(\mathcal A)$ and $\mathcal D(\mathrm{Flat}(A))$ must be equivalent.
\end{proof}

Note that the two injective cotorsion pairs in the above proof may each be thought of as the ``DG-injective'' cotorsion pairs, but with respect to their exact structure. Similarly the projective cotorsion pairs may be thought of as the ``DG-projective'' cotorsion pairs with respect to these exact structures.

So it seems clear that in order to gain a better understanding of the pure derived category, one should focus on the derived category of flat modules. In \cite{MS} Murfet and Salarian define the {\it pure derived category of flat sheaves} for a semi-separated noetherian scheme. But a close inspection of their definition reveals that they are considering the {\it derived category of flat sheaves} in the above sense. The next result shows that, for \emph{any} scheme $X$, we can realize the derived category of flat sheaves as the homotopy category of a model structure on $\Ch({\rm Flat}(X))$ which is injective with respect to the exact structure. 

\begin{theorem}
Let $X$ be a scheme, and ${\rm Flat}(X)$ the category of quasi-coherent flat sheaves. There is an injective exact model structure on $\Ch({\rm Flat}(X))$.  So every object is cofibrant and the fibrant objects are dg-cotorsion complexes which are flat on each degree. The trivial objects are those in $\Ch_{ac}({\rm Flat}(X)) = \widetilde{\mathcal{F}}$, the class of acyclic complexes with flat cycles. The corresponding homotopy category is the derived category of flat sheaves, $\mathcal D(\mathrm{Flat}(X))$.
\end{theorem}

\begin{proof}
The category of quasi-coherent sheaves is Grothendieck and the class ${\rm Flat}(X)$ of flat quasi-coherent sheaves is deconstructible. So by~\cite[Theorem~3.16]{J} we get that ${\rm Flat}(X)$ inherits the structure of an exact category of Grothendieck type. Moreover, by~\cite[Lemma~7.9]{J} we get that $\Ch_{ac}({\rm Flat}(X)) = \widetilde{\mathcal{F}}$ is deconstructible in the exact category $\Ch({\rm Flat}(X))$. Then we conclude, using~\cite[Theorem~7.11]{J},  that $(\widetilde{\mathcal{F}}, \widetilde{\mathcal{F}}^{\perp})$ is an injective model structure in  $\Ch({\rm Flat}(X))$. 
It is left to argue that $\widetilde{\mathcal{F}}^{\perp} = \mathcal{Y} \cap \Ch_{ac}({\rm Flat}(X))$ where $\mathcal{Y}$ is the class of dg-cotorsion complexes in $\Ch(\Qco(X))$. We are calling a complex $Y \in \Ch(\Qco(X))$ \emph{dg-cotorsion} if each $Y_n$ is cotorsion and  every chain map $F \xrightarrow{} Y$ is null homotopic whenever $F$ is in $\widetilde{\mathcal{F}}$. So then $\mathcal{Y} \cap \Ch_{ac}({\rm Flat}(X))$ is the class of complexes $Y$ with each $Y_n$ cotorsion flat and with every chain map $F \xrightarrow{} Y$ being null homotopic whenever $F$ is in $\widetilde{\mathcal{F}}$. But now using that the injective objects in ${\rm Flat}(X)$ are the cotorsion flats, we can argue as in Corollary~\ref{cor-DG-pure-injective cot pair} that this coincides with  $\widetilde{\mathcal{F}}^{\perp}$ in $\Ch({\rm Flat}(X))$. 
\end{proof}

\end{document}